\documentclass[12pt,leqno]{article}
\usepackage{hyperref}
\usepackage{xcolor}
\definecolor{labelkey}{rgb}{0,0.08,0.45}
\definecolor{refkey}{rgb}{0,0.6,0.0}
\definecolor{Brown}{rgb}{0.45,0.0,0.05}
\usepackage{exscale,relsize}
\usepackage{amsmath}
\usepackage{amsfonts}
\usepackage{amssymb}
\usepackage{calc}
\usepackage{theorem}
\usepackage{pifont}      
\usepackage{graphicx}
\newcommand{\emp}{\ensuremath{\varnothing}}

\newcommand{\scal}[2]{\langle{{#1},{#2}}\rangle}

\newcommand{\RR}{\ensuremath{\mathbb R}}

\newcommand{\RPP}{\ensuremath{\,\left]0,+\infty\right[}}
\newcommand{\RX}{\ensuremath{\,\left]-\infty,+\infty\right]}}
\newcommand{\RXX}{\ensuremath{\,\left[-\infty,+\infty\right]}}

\newcommand{\menge}[2]{\big\{{#1} \mid {#2}\big\}}

\newcommand{\To}{\ensuremath{\rightrightarrows}}

\newcommand{\dom}{\ensuremath{\operatorname{dom}}}

\newcommand{\gra}{\ensuremath{\operatorname{gra}}}

\newcommand{\intdom}{\ensuremath{\operatorname{int}\operatorname{dom}}\,}
\newcommand{\inte}{\ensuremath{\operatorname{int}}}

\newcommand{\pinf}{\ensuremath{+\infty}}

\renewcommand{\phi}{\ensuremath{\varphi}}

\newtheorem{theorem}{Theorem}[section]
\newtheorem{lemma}[theorem]{Lemma}
\newtheorem{fact}[theorem]{Fact}

\theoremstyle{plain}{\theorembodyfont{\rmfamily}
}
\theoremstyle{plain}{\theorembodyfont{\rmfamily}
}
\theoremstyle{plain}{\theorembodyfont{\rmfamily}
}
\theoremstyle{plain}{\theorembodyfont{\rmfamily}
}
\theoremstyle{plain}{\theorembodyfont{\rmfamily}
\newtheorem{remark}[theorem]{Remark}}
\theoremstyle{plain}{\theorembodyfont{\rmfamily}
}

\begin{document}


\title{\textsc{
On the maximal monotonicity of the sum \\ of a maximal monotone
linear relation\\ and the subdifferential operator \\
of a sublinear function }}

\author{
Heinz H.\ Bauschke\thanks{Mathematics, Irving K.\ Barber School,
UBC Okanagan, Kelowna, British Columbia V1V 1V7, Canada. E-mail:
\texttt{heinz.bauschke@ubc.ca}.},\;  Xianfu
Wang\thanks{Mathematics, Irving K.\ Barber School, UBC Okanagan,
Kelowna, British Columbia V1V 1V7, Canada. E-mail:
\texttt{shawn.wang@ubc.ca}.},\; and Liangjin\
Yao\thanks{Mathematics, Irving K.\ Barber School, UBC Okanagan,
Kelowna, British Columbia V1V 1V7, Canada.
E-mail:  \texttt{ljinyao@interchange.ubc.ca}.}}
 \vskip 3mm

\date{January 1, 2010}
\maketitle

\begin{abstract} \noindent
The most important open problem in Monotone Operator Theory
concerns the maximal monotonicity of the sum of two
maximal monotone operators provided that
Rockafellar's constraint qualification holds.

In this note, we provide a new maximal monotonicity result for the sum
of a maximal monotone relation and the subdifferential operator of
a proper, lower semicontinuous, sublinear function.
The proof relies on Rockafellar's formula for the Fenchel conjugate of the
sum as well as some results on the Fitzpatrick function.
\end{abstract}



\section{Introduction}

Throughout this paper, we assume that
$X$ is a real Banach space with norm $\|\cdot\|$,
that $X^*$ is the continuous dual of $X$, and
that $X$ and $X^*$ are paired by $\scal{\cdot}{\cdot}$.
Let $A\colon X\To X^*$
be a \emph{set-valued operator} (also known as multifunction)
from $X$ to $X^*$, i.e., for every $x\in X$, $Ax\subseteq X^*$,
and let
$\gra A = \menge{(x,x^*)\in X\times X^*}{x^*\in Ax}$ be
the \emph{graph} of $A$.
Recall that $A$ is  \emph{monotone} if
\begin{equation}
\big(\forall (x,x^*)\in \gra A\big)\big(\forall (y,y^*)\in\gra
A\big) \quad \scal{x-y}{x^*-y^*}\geq 0,
\end{equation}
and \emph{maximal monotone} if $A$ is monotone and $A$ has no proper monotone extension
(in the sense of graph inclusion).
We say $A$ is a \emph{linear relation} if $\gra A$ is a linear subspace.
Monotone operators have proven to be a key class of objects
in modern Optimization and Analysis; see, e.g.,
the books
\cite{BorVan,BurIus,ButIus,ph,Si,Si2,RockWets,Zalinescu}
and the references therein.
(We also adopt standard notation used in these books:
$\dom A = \menge{x\in X}{Ax\neq\varnothing}$ is the \emph{domain} of $A$.
Given a subset $C$ of $X$,
$\inte C$ is the \emph{interior} of $C$, and
$\overline{C}$ is the \emph{closure} of $C$. We set $C^{\bot}:=
\{x^*\in X^* \mid(\forall c\in C)\, \langle x^*, c\rangle=0\}$
and $S^{\bot}:=
\{x^{**}\in X^{**} \mid(\forall s\in S)\, \langle x^{**}, s\rangle=0\}$ for a set  $S\subseteq X^*$.
The \emph{indicator function} of $C$, written as $\iota_C$, is defined
at $x\in X$ by
\begin{align}
\iota_C (x):=\begin{cases}0,\,&\text{if $x\in C$;}\\
\infty,\,&\text{otherwise}.\end{cases}\end{align}
 Given $f\colon X\to \RX$, we set
$\dom f = f^{-1}(\RR)$ and
$f^*\colon X^*\to\RXX\colon x^*\mapsto\sup_{x\in X}(\scal{x}{x^*}-f(x))$
is the \emph{Fenchel conjugate} of $f$.
If $f$ is convex and $\dom f\neq\varnothing$, then
   $\partial f\colon X\To X^*\colon
   x\mapsto \menge{x^*\in X^*}{(\forall y\in
X)\; \scal{y-x}{x^*} + f(x)\leq f(y)}$
is the \emph{subdifferential operator} of $f$.
Recall that $f$ is \emph{sublinear}
if $f(0)=0$,
$f(x+y)\leq f(x)+f(y)$, and
$f(\lambda x)=\lambda f(x)$ for all $x,y\in \dom f$ and $\lambda>0$.
Finally,  the \emph{closed unit ball} in $X$ is denoted by
$B_X := \menge{x\in X}{\|x\|\leq 1}$.)
Throughout, we shall identify $X$ with its canonical image in the
bidual space $X^{**}$.
Furthermore, $X\times X^*$ and $(X\times X^*)^* = X^*\times X^{**}$ are
likewise paired via $\scal{(x,x^*)}{(y^*,y^{**})} = \scal{x}{y^*} +
\scal{x^*}{y^{**}}$, where $(x,x^*)\in X\times X^*$ and $(y^*,y^{**}) \in
X^*\times X^{**}$.

Let $A$ and $B$ be maximal monotone operators from $X$ to
$X^*$.
Clearly, the \emph{sum operator} $A+B\colon X\To X^*\colon x\mapsto
Ax+Bx = \menge{a^*+b^*}{a^*\in Ax\;\text{and}\;b^*\in Bx}$
is monotone.
Rockafellar's \cite[Theorem~1]{Rock70} guarantees maximal monotonicity
of $A+B$ under
the classical \emph{constraint qualification}
$\dom A \cap\intdom B\neq \varnothing$ when $X$ is reflexive.
The most famous open problem concerns the behaviour in
nonreflexive Banach spaces. See Simons' monograph
\cite{Si2} for a comprehensive account of the recent developments.

Now we focus on the special case when
$A$ is a \emph{linear relation}
and $B$ is the subdifferential operator
of a \emph{sublinear} function $f$. We show that the sum theorem is true
in this setting.
Recently, linear relations have increasingly
been studied in detail; see, e.g.,
\cite{BB,BBW,BWY2,BWY3,BWY4, BWY7, BWY8, PheSim,Svaiter,Voisei06,VZ,Yao}
and Cross' book \cite{Cross} for general background on linear
relations.

The remainder of this paper is organized as follows.
In Section~\ref{s:aux}, we collect auxiliary results for future reference
and for the
reader's convenience.
The main result (Theorem~\ref{t:main}) is proved
in Section~\ref{s:main}.

\section{Auxiliary Results}
\label{s:aux}

\begin{fact}[Rockafellar] \label{f:F4}
\emph{(See {\cite[Theorem~3]{Rock66}},
{\cite[Corollary~10.3 and Theorem~18.1]{Si2}}, or
{\cite[Theorem~2.8.7(iii)]{Zalinescu}}.)}\\
Let $f,g: X\rightarrow\RX$ be proper convex functions.
Assume that there exists a point $x_0\in\dom f \cap \dom g$
such that $g$ is continuous at $x_0$.
Then for every $z^*\in X^*$,
there exists $y^*\in X^*$ such that
\begin{equation}
(f+g)^*(z^*) = f^*(y^*)+g^*(z^*-y^*).
\end{equation} Furthermore, $\partial (f+g)=\partial f+\partial g$.
\end{fact}

\begin{fact}[Fitzpatrick]
\emph{(See {\cite[Corollary~3.9]{Fitz88}}.)}
\label{f:Fitz}
Let $A\colon X\To X^*$ be maximal monotone,  and set
\begin{equation}
F_A\colon X\times X^*\to\RX\colon
(x,x^*)\mapsto \sup_{(a,a^*)\in\gra A}
\big(\scal{x}{a^*}+\scal{a}{x^*}-\scal{a}{a^*}\big),
\end{equation}
which is the \emph{Fitzpatrick function} associated with $A$.
Then for every $(x,x^*)\in X\times X^*$, the inequality
$\scal{x}{x^*}\leq F_A(x,x^*)$ is true,
and equality holds if and only if $(x,x^*)\in\gra A$.
\end{fact}

\begin{fact}[Simons]
\emph{(See \cite[Theorem~24.1(c)]{Si2}.)}
\label{f:referee1}
Let $A, B:X\To X^*$ be maximal monotone operators. Assume
$\bigcup_{\lambda>0} \lambda\left[P_X(\dom F_A)-P_X(\dom F_B)\right]$
is a closed subspace, where $P_X:(x,x^*)\in X\times X^*\rightarrow x$.
If
\begin{equation}
(x,x^*)\,\text{is monotonically related to $\gra (A+B)$}
\Rightarrow x\in\dom A\cap\dom B,
\end{equation}
then $A+B$ is maximal monotone.
\end{fact}

\begin{fact}[Simons]
\emph{(See \cite[Lemma~19.7 and Section~22]{Si2}.)}
\label{f:referee}
Let $A:X\To X^*$ be a monotone linear relation such that $\gra A
\neq\varnothing$.
Then the function
\begin{equation}
g\colon X\times X^* \to \RX\colon
(x,x^*)\mapsto \scal{x}{x^*} + \iota_{\gra A}(x,x^*)
\end{equation}
is proper and convex.
\end{fact}

\begin{fact}[Simons]
\emph{(See \cite[Lemma~2.2]{Si3}.)}
\label{slope}
Let $f:X\to\RX$ be proper, lower semicontinuous, and convex.
Let $x\in X$ and $\lambda\in\RR$ be such that
$\inf f<\lambda<f(x) \leq\pinf$, and set
\begin{align*}
K:=\sup_{a\in X, a\neq x}\frac{\lambda-f(a)}{\|x-a\|}.\end{align*}
Then $K\in\RPP$ and
for every $\varepsilon\in\left]0,1\right[$,
there exists $(y,y^*)\in \gra \partial f$ such that
\begin{align}
\langle y-x,y^*\rangle\leq-(1-\varepsilon)K\|y-x\|<0.\end{align}
\end{fact}

\begin{fact}
\emph{(See \cite[Therorem~2.4.14]{Zalinescu}.)}
\label{zls:1}
Let $f:X\to \RX$ be a sublinear function.
Then the following hold.
\begin{enumerate}
\item\label{zf:2} $\partial f(x)=\{x^*\in \partial f(0)\mid \langle x^*,x\rangle=f(x)\},\quad
\forall x\in \dom f$.
\item\label{zf:3} $\partial f(0)\neq\varnothing\Leftrightarrow$ $f$ is lower semicontinuous at $0$.
\item\label{zf:4}If $f$ is lower semicontinuous, then
$f=\sup\langle\cdot,\partial f(0)\rangle$.
\end{enumerate}
\end{fact}
\begin{fact}\emph{(See \cite[Proposition~3.3 and Proposition~1.11]{ph}.)}
\label{pheps:1}Let $f:X\rightarrow\RX$ be a lower semicontinuous convex
 and $\intdom f\neq\varnothing$.
Then $f$ is continuous on $\intdom f$ and $\partial f(x)\neq\varnothing$ for every $x\in\intdom f$.
\end{fact}

\begin{lemma}\label{rcf:1}
Let $f:X\to\RX$ be a sublinear function.
Then $\dom f+\intdom f=\intdom f$.
\end{lemma}

\begin{proof}
The result is trivial when $\intdom f=\emp$ so we assume that
$x_0\in\intdom f$.
Then there exists $\delta>0$ such that $x_0+ \delta B_X\subseteq\dom f$.
By sublinearity,
$\forall y\in\dom f$, we have $y+x_0+ \delta B_X\subseteq\dom f$.
Hence \begin{align*}y+x_0\in\inte\dom f.\end{align*}
Then
  $\dom f+\intdom f\subseteq\intdom f$.
  Since $0\in\dom f$,   $\intdom f\subseteq\dom f+\intdom f$. Hence $\dom f+\intdom f=\intdom f$.
\end{proof}

\begin{lemma}\label{rcf:001}
Let $A:X\To X^*$ be a maximal monotone linear relation,
and let $z\in X\cap (A0)^\bot$.
Then $z\in\overline{\dom A}$.
\end{lemma}
\begin{proof}
Suppose to the contrary that
$z\notin \overline{\dom A}$.
Then the Separation Theorem provides $w^*\in X^*$ such that
\begin{align}
\label{ree:3}
\langle z, w^*\rangle>0
\quad\text{and}\quad
w^*\in \overline{\dom A}^\bot.
\end{align}
Thus, $(0, w^*)$ is monotonically related to $\gra A$.
Since $A$ is maximal monotone, we deduce that
$w^*\in A0$. By assumption,
$\langle z, w^*\rangle=0$, which contradicts \eqref{ree:3}.
Hence, $z\in\overline{\dom A}$.
\end{proof}

The proof of the next result follows closely the proof of
\cite[Theorem~53.1]{Si2}.

\begin{lemma}\label{rcf:01}
Let $A:X\To X^*$ be a  monotone linear relation,
and let $f:X\to\RX$ be a proper lower semicontinuous convex function.
Suppose that $\dom A \cap \inte \dom \partial f\neq \varnothing$,
$(z,z^*)\in X\times X^*$ is monotonically related to
$\gra (A+\partial f)$, and that $z\in\dom A$.
Then  $z\in  \dom \partial f$.
\end{lemma}
\begin{proof}
Let $c_0\in X$ and $y^*\in X^*$ be such that
\begin{align}c_0\in \dom A \cap \inte \dom \partial f
\quad\text{and}\quad(z,y^*)\in\gra A.\label{ree:e1}\end{align}
Take $c_0^*\in Ac_0$, and set
\begin{align}
\label{ree:8}
M:=\max\big\{\|y^*\|,\|c^*_0\|\big\},
\end{align}
$D:=\left[c_0,z\right]$, and $h:=f+\iota_D$.
By \eqref{ree:e1}, Fact~\ref{pheps:1} and Fact~\ref{f:F4}, $\partial h=\partial f+\partial\iota_D$.
Set $H\colon X\to\RX\colon x \mapsto h(x+z)-\scal{z^*}{x}$.
It remains to show that
\begin{align}
0\in\dom\partial H.\label{ree:6}\end{align}
If $\inf H=H(0)$, then \eqref{ree:6} holds.
Now suppose that $\inf H<H(0)$.
Let $\lambda\in\RR$ be such that $\inf H<\lambda<H(0)$, and set
\begin{align}K_{\lambda}:=\sup_{H(x)<\lambda}\frac{\lambda-H(x)}{\|x\|}.\label{ree:18}\end{align}
We claim that
\begin{align*}
K_{\lambda}\leq M.
\end{align*} By Fact~\ref{slope},
we have $K_{\lambda}\in\left]0,\infty\right[$ and $\forall\varepsilon\in\left]0,1\right[$, by $\gra \partial H=\gra \partial h-(z,z^*)$  there
exists $(x,x^*)\in \gra \partial h$ such that
\begin{align}
\langle x-z,x^*-z^*\rangle\leq-(1-\varepsilon)K_{\lambda}\|x-z\|<0.\label{ree:7}\end{align}
Since $\partial h=\partial f+\partial\iota_D$, there exists $t\in\left[0,1\right]$ with
$x^*_1\in \partial f(x)$ and $x^*_2\in \partial\iota_D(x)$ such that $x=tc_0+(1-t)z$ and
$x^*=x^*_1+x^*_2$. Then $\langle x-z,x^*_2\rangle\geq0$.
Thus, by \eqref{ree:7},
\begin{align}
\langle x-z,x^*_1-z^*\rangle\leq\langle x-z,x^*_1+x^*_2-z^*\rangle\leq-(1-\varepsilon)K_{\lambda}\|x-z\|<0.\label{ree:10}\end{align}
As $x=tc_0+(1-t)z$ and $A$ is a linear relation, we have $(x,tc^*_0+(1-t)y^*)\in \gra A$.
 Since $(z,z^*)$ is monotonically related to $\gra (A+\partial f)$, by \eqref{ree:8},
\begin{align}
\langle x-z,x^*_1-z^*\rangle\geq-\langle x-z, tc^*_0+(1-t)y^*\rangle
\geq-M \|x-z\|.\label{ree:9}\end{align}
Combining \eqref{ree:9} and \eqref{ree:10}, we obtain
\begin{align}
-M \|x-z\|\leq-(1-\varepsilon)K_{\lambda}\|x-z\|<0.\label{ree:11}\end{align}
Hence, $(1-\varepsilon)K_{\lambda}\leq M$.
Letting $\varepsilon\downarrow0$,
we deduce that $K_{\lambda}\leq M$.
Then, by \eqref{ree:18} and letting $\lambda\uparrow H(0)$,
we get
\begin{align}
H(y)+M\|y\|\geq H(0),\quad \forall y\in X.\end{align}
By \cite[Example~7.1]{Si2}, $0\in \dom \partial H$.
Hence \eqref{ree:6} holds and thus $z\in\dom\partial f$.
\end{proof}

\section{Main Result}
\label{s:main}

\begin{theorem}\label{t:main}
Let $A:X\To X^*$ be a maximal monotone linear relation,
let $f:X\rightarrow\RX$ be a proper lower semicontinuous sublinear function,
and suppose that $\dom A \cap \inte \dom \partial f\neq \varnothing$.
Then $A+\partial f$ is maximal monotone.
\end{theorem}

\begin{proof}
Let $(z,z^*)\in X\times X^*$ and suppose that
\begin{equation}
\text{
$(z,z^*)$ is monotonically related to $\gra(A+\partial f)$.
}
\end{equation}
By Fact~\ref{f:Fitz}, $\dom A\subseteq P_X(F_A)$ and $\dom \partial f\subseteq P_X(F_{\partial f})$.
Hence,
\begin{align}\bigcup_{\lambda>0} \lambda
\big(P_X(\dom F_A)-P_X(\dom F_{\partial f})\big)=X.\end{align}
Thus, by Fact~\ref{f:referee1}, it suffices to show that
\begin{equation} \label{e:ourgoal}
z\in\dom A\cap\dom \partial f.
\end{equation}
We have
\begin{align}
&\langle z, z^*\rangle-\langle z,x^*\rangle-\langle x,z^*\rangle
+\langle x,x^*\rangle+\langle x-z,y^*\rangle\nonumber\\
&=\langle z-x, z^*-x^*-y^*\rangle\geq0, \quad \forall
(x,x^*)\in\gra A, (x,y^*)\in\gra \partial f.\label{ree:1}
\end{align}
By Fact~\ref{zls:1}\ref{zf:3}, $\partial f(0)\neq\varnothing$.
By \eqref{ree:1},
\begin{align*}
\inf \left[\langle z, z^*\rangle-\langle z,A0\rangle
-\langle z,\partial f(0)\rangle\right]\geq 0.
\end{align*}
Thus, \begin{align}z\in X\cap (A0)^\bot.\label{ree:2}\end{align}
Then, by Fact~\ref{zls:1}\ref{zf:4},
\begin{align*}
\langle z, z^*\rangle\geq f(z).
\end{align*}
Thus, \begin{align}z\in\dom f.\label{ree:15}\end{align}
By \eqref{ree:2} and Lemma~\ref{rcf:001}, we have
\begin{align}z\in\overline{\dom A}.\label{ree:5}\end{align}
By Fact~\ref{zls:1}\ref{zf:2}, $y^*\in \partial f(0)$ as  $y^*\in \partial f(x)$.
Then $\langle x-z,y^*\rangle\leq f(x-z),\quad \forall y^*\in \partial f(x)$.
Thus, by \eqref{ree:1}, we have
\begin{align}
&\langle z, z^*\rangle-\langle z,x^*\rangle-\langle x,z^*\rangle
+\langle x,x^*\rangle+f(x-z)\geq0,
 \quad \forall
(x,x^*)\in\gra A, x\in\dom \partial f.
\end{align}
Let $C:=\intdom f$.
Then by Fact~\ref{pheps:1}, we have
\begin{align}
&\langle z, z^*\rangle-\langle z,x^*\rangle-\langle x,z^*\rangle
+\langle x,x^*\rangle+f(x-z)\geq0,
 \quad \forall
(x,x^*)\in\gra A, x\in C.\label{e:zz*}
\end{align}
Set  $j:=(f(\cdot-z)+\iota_{C}) \oplus\iota_{X^*}$ and
\begin{equation}\label{e:defofg}
g\colon X\times X^* \to \RX\colon
(x,x^*)\mapsto \scal{x}{x^*} + \iota_{\gra A}(x,x^*).
\end{equation}
By Fact~\ref{f:referee}, $g$ is convex.
Hence,
\begin{equation} \label{e:defofh}
h:= g + j
\end{equation}
is convex as well.
Let
\begin{equation} \label{e:defofc0}
c_0 \in \dom A \cap C.
\end{equation}
By Lemma~\ref{rcf:1} and \eqref{ree:15}, $z+c_0\in \intdom f$. Then there exists $\delta>0$ such that $z+c_0+ \delta B_X\subseteq\dom f$
and $c_0+ \delta B_X\subseteq\dom f$.
By \eqref{ree:5}, $z+c_0\in\overline{\dom A}$ since $\dom A$ is a linear subspace. Thus there exists $b\in \tfrac{1}{2}\delta B_X$ such that
$z+c_0+b\in\dom A\cap \inte\dom f$. Let $v^*\in A(z+c_0+b)$. Since $c_0+b\in\inte\dom f$,
\begin{align}(z+c_0+b,v^*)\in \gra A \cap
\big(\inte C\cap\inte \dom f(\cdot-z) \times X^*\big)
=\dom g\cap \inte\dom j\neq\varnothing.\end{align}
By Fact~\ref{f:F4} and Fact~\ref{pheps:1},
there exists $(y^*,y^{**})\in X^{*}\times X^{**}$ such that
\begin{align}
h^*(z^*,z) &= g^*(y^*,y^{**}) + j^*(z^*-y^*,z-y^{**})\nonumber\\
&=g^*(y^*,y^{**})+ \iota_{\{0\}}(z-y^{**})
+ \sup_{x\in C}\left[\langle x,  z^*-y^*\rangle-f(x-z)\right]\nonumber\\
&\geq g^*(y^*,y^{**})+ \iota_{\{0\}}(z-y^{**})
+ \sup_{x\in z+ C}\left[\langle x,  z^*-y^*\rangle-f(x-z)\right]\,\text{(by Lemma~\ref{rcf:1} and \eqref{ree:15})}\nonumber\\
&= g^*(y^*,y^{**})+ \iota_{\{0\}}(z-y^{**})+\langle z,z^*-y^*\rangle
+ \sup_{y\in C}\left[\langle y,  z^*-y^*\rangle-f(y)\right]\nonumber\\
&= g^*(y^*,y^{**})+ \iota_{\{0\}}(z-y^{**})+\langle z,z^*-y^*\rangle
+ \sup_{\{y\in C, k>0\}}\left[\langle ky,  z^*-y^*\rangle-f(ky)\right]\nonumber\\
&= g^*(y^*,y^{**})+ \iota_{\{0\}}(z-y^{**})+\langle z,z^*-y^*\rangle
+\sup_{\{y\in C,k>0\}}k\left[\langle y,  z^*-y^*\rangle-f(y)\right]\nonumber\\
&\geq g^*(y^*,y^{**})+ \iota_{\{0\}}(z-y^{**})+\langle z,z^*-y^*\rangle.\label{e:puncha}
\end{align}
By \eqref{e:zz*}, we have,
for every $(x,x^*)\in\gra A \cap (C\times X^*)$,
$\scal{(x,x^*)}{(z^*,z)} - h(x,x^*) =
\scal{x}{z^*}+\scal{z}{x^*}-\scal{x}{x^*}-f(x-z) \leq \scal{z}{z^*}$.
Consequently,
\begin{equation} \label{e:punchb}
h^*(z^*,z) \leq \scal{z}{z^*}.
\end{equation}
Combining \eqref{e:puncha} with \eqref{e:punchb}, we obtain
\begin{equation} \label{e:punchc}
g^*(y^*,y^{**})+\langle z,z^*-y^*\rangle  + \iota_{\{0\}}(z-y^{**}) \leq
\scal{z}{z^*}.
\end{equation}
Therefore, $y^{**}=z$. Hence
$g^*(y^*,z) + \langle z,z^*-y^*\rangle \leq \scal{z}{z^*}$.
Since $g^*(y^*,z) = F_A(z,y^*)$, we deduce that
$F_A(z,y^*)\leq \scal{z}{y^*}$.
By Fact~\ref{f:Fitz},
\begin{equation}
\label{e:punchx}
(z,y^*)\in\gra A
\end{equation}
Hence \begin{align*} z\in\dom A.\end{align*}
Apply  Lemma~\ref{rcf:01} to obtain $z\in\dom\partial f$.
Then $z\in\dom A\cap\dom\partial f$.
Hence $A+B$ is maximal monotone.
\end{proof}

\begin{remark}
Verona and Verona 
(see \cite[Corollary~2.9(a)]{VV} or \cite[Theorem~53.1]{Si2}) 
showed the following: 
``Let $f: X\to \RX$ be proper, lower semicontinuous, and convex, 
let $A: X\To X^*$ be maximal monotone, and suppose that
dom $A=X$. 
Then $\partial f +A$ is maximal monotone.''
Note that 
Theorem~\ref{t:main} cannot be deduced from this result 
because $\dom A$ need not have full domain.
\end{remark}

\section*{Acknowledgment}
Heinz Bauschke was partially supported by the Natural Sciences and
Engineering Research Council of Canada and
by the Canada Research Chair Program.
Xianfu Wang was partially supported by the Natural
Sciences and Engineering Research Council of Canada.

\end{document}